\documentclass[11pt]{article}
\usepackage[T1]{fontenc}
\usepackage{lmodern}
\usepackage[a4paper,margin=25mm]{geometry}
\usepackage{amsmath,amssymb,amsthm,mathtools}
\usepackage[dvipsnames]{xcolor}
\usepackage{booktabs,array,microtype,tikz}
\usetikzlibrary{arrows.meta}
\usepackage[colorlinks=true,linkcolor=MidnightBlue,citecolor=MidnightBlue,urlcolor=MidnightBlue]{hyperref}
\usepackage{fancyhdr}
\numberwithin{equation}{section}
\newtheorem{theorem}{Theorem}[section]
\newtheorem{lemma}[theorem]{Lemma}
\newtheorem{proposition}[theorem]{Proposition}
\newtheorem{corollary}[theorem]{Corollary}
\theoremstyle{definition}
\newtheorem{definition}[theorem]{Definition}

\newcommand{\one}{\mathbf 1}
\newcommand{\C}{\mathbb C}
\newcommand{\Spec}{\operatorname{Spec}}
\newcommand{\norm}[1]{\left\lVert #1\right\rVert}

\title{Aleshin graphs are expanders}
\author{Olga Kharlampovich and Alina Vdovina}
\date{20 October 2026}
\hypersetup{pdftitle={Aleshin graphs are expanders},pdfsubject={Spectral expansion and linear diameter bounds for Aleshin and Bellaterra graphs},pdfkeywords={Aleshin, Bellaterra, expander, Schreier graph, spectral gap, diameter, Chebyshev polynomial}}

\begin{document}
\maketitle
\begin{abstract}
Using the Bellaterra Ramanujan theorem \cite{Bell}, we prove that the
Aleshin level graphs are two-sided expanders and establish linear
diameter bounds for both graph families. At level $n$, the diameters
of the Bellaterra and undirected Aleshin graphs are at most $2n+2$,
and the directed Aleshin diameter is at most $12n$. Together with
elementary lower bounds, these estimates show that all three diameter
sequences have growth $\Theta(n)=\Theta(\log |V_n|)$, where
$|V_n|=2^n$. This improves the quadratic diameter bounds of Malyshev
and Pak to the correct order of growth. For the Aleshin graphs we
obtain a normalized spectral gap of at least $1-2\sqrt2/3$ and an
edge-expansion constant of at least $3-2\sqrt2$, thereby establishing
Malyshev--Pak's Conjecture~1.2 with explicit constants.
An exact calculation at level ten also proves that the undirected
Aleshin graphs are not Ramanujan for any $n\geq10$.
\end{abstract}

\section{The automata and their level graphs}

Malyshev and Pak proved that the diameters of the Aleshin and
Bellaterra level graphs grow at most quadratically in the level
\cite[Theorem~1.1]{MP}. They conjectured that both families are
two-sided expanders \cite[Conjectures~1.2 and~1.3]{MP}.
The connection between these questions is that uniform spectral
expansion gives logarithmic diameter in the number of vertices.
Here the number of vertices is $2^n$, so logarithmic diameter means
linear growth in $n$. We prove the expansion transfer first and then
give explicit linear diameter bounds in Section~\ref{sec:diameters}.
We note that Ramanujan graphs of valence $q+1$
with $q$ an odd prime power, generated by automata, were contructed by Bondarenko, Grigorchuk and Vdovina
in \cite[Theorem~2]{BGV}.

\subsection{Mealy automata and binary words}

Let $X=\{0,1\}$. The vertices of the rooted binary tree $X^*$ are
the finite words in $X$, including the empty word, which is the root.
The two children of a word $w$ are $w0$ and $w1$. The $n$th level is
$V_n=X^n$, so $|V_n|=2^n$.

A \emph{Mealy automaton} consists of a finite set of states, an
input alphabet, a next-state rule, and an output rule. Starting in a
state $s$, it reads a word from left to right. At each step it outputs
one letter and changes state according to the current state and the
input letter. Thus a state defines a length-preserving map on words.
If the output rule at each state permutes the alphabet, these maps
are automorphisms of the rooted tree. In particular, they permute
each finite level $V_n$.

For a rooted binary-tree automorphism $g$, the notation
$g=(g_0,g_1)\pi$ means
\[
 g(xw)=\pi(x)g_x(w)\qquad(x\in X).
\]
Here $\pi$ is the permutation of the first letter, and $g_x$ is the
\emph{section} of $g$ at $x$, namely the action on the remaining
word after the first letter $x$ has been read. Put $\sigma=(0\ 1)$.
The displayed formulas below are called \emph{wreath recursions}.
They specify each automaton completely.

\subsection{The Aleshin and Bellaterra automata}

We use the Aleshin convention of Vorobets--Vorobets \cite{VV}:
\begin{equation}\label{eq:aleshin}
 a=(c,b)\sigma,\qquad
 b=(b,c)\sigma,\qquad
 c=(a,a).
\end{equation}
The matching Bellaterra states are
\begin{equation}\label{eq:bellaterra}
 \beta_a=(\beta_c,\beta_b),\qquad
 \beta_b=(\beta_b,\beta_c),\qquad
 \beta_c=(\beta_a,\beta_a)\sigma.
\end{equation}
For example, $a(0w)=1c(w)$ and $a(1w)=0b(w)$, whereas
$\beta_a(0w)=0\beta_c(w)$ and $\beta_a(1w)=1\beta_b(w)$.
The next-state rules agree, and the output letters are opposite.
Figure~\ref{fig:automata} displays both automata. An arrow labelled
$x/y$ means: read $x$, output $y$, and move to the state at the
end of the arrow. Two labels on one arrow specify two transitions
with the same next state.

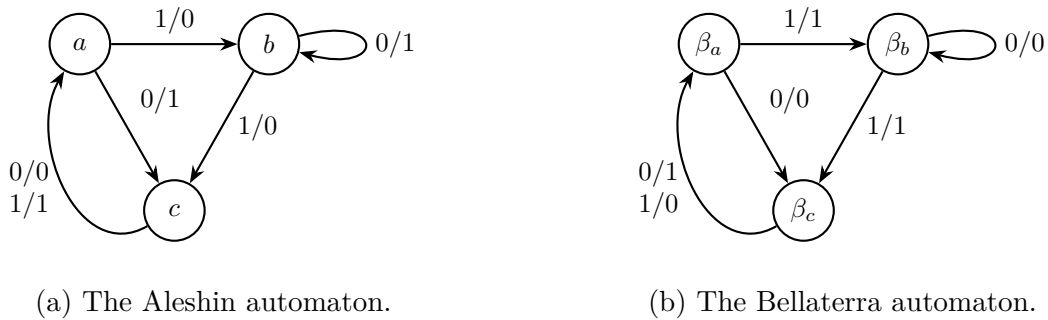
\begin{figure}[htbp]
\centering
\begin{minipage}[t]{0.48\textwidth}
\centering
\begin{tikzpicture}[
 >=Stealth,thick,
 state/.style={circle,draw,minimum size=8mm,inner sep=1pt},
 every edge/.style={draw,->},
 every node/.style={font=\small}]
 \node[state] (a) at (0,1.4) {$a$};
 \node[state] (b) at (2.5,1.4) {$b$};
 \node[state] (c) at (1.25,-0.8) {$c$};
 \path
 (a) edge node[above] {$1/0$} (b)
 (a) edge node[above right=1pt] {$0/1$} (c)
 (b) edge[loop right,min distance=12mm] node[right] {$0/1$} (b)
 (b) edge node[right=2pt] {$1/0$} (c)
 (c) edge[out=210,in=240,looseness=1.15]
       node[left=2pt,align=center] {$0/0$\\$1/1$} (a);
\end{tikzpicture}
\par\smallskip
\textup{(a) The Aleshin automaton.}
\end{minipage}\hfill
\begin{minipage}[t]{0.48\textwidth}
\centering
\begin{tikzpicture}[
 >=Stealth,thick,
 state/.style={circle,draw,minimum size=8mm,inner sep=1pt},
 every edge/.style={draw,->},
 every node/.style={font=\small}]
 \node[state] (a) at (0,1.4) {$\beta_a$};
 \node[state] (b) at (2.5,1.4) {$\beta_b$};
 \node[state] (c) at (1.25,-0.8) {$\beta_c$};
 \path
 (a) edge node[above] {$1/1$} (b)
 (a) edge node[above right=1pt] {$0/0$} (c)
 (b) edge[loop right,min distance=12mm] node[right] {$0/0$} (b)
 (b) edge node[right=2pt] {$1/1$} (c)
 (c) edge[out=210,in=240,looseness=1.15]
       node[left=2pt,align=center] {$0/1$\\$1/0$} (a);
\end{tikzpicture}
\par\smallskip
\textup{(b) The Bellaterra automaton.}
\end{minipage}
\caption{The two automata in the conventions
\eqref{eq:aleshin} and \eqref{eq:bellaterra}. The directed edges
record next states; each label gives the input letter followed by
the output letter.}
\label{fig:automata}
\end{figure}

As a short example of reading a word,
\[
 a(010)=1c(10)=11a(0)=111,
 \qquad
 \beta_a(010)=0\beta_c(10)=00\beta_a(0)=000.
\]
In each calculation the remaining input is processed from the state
specified by the preceding transition.

The Bellaterra states are involutions. Indeed, their squares satisfy
\[
 \beta_a^2=(\beta_c^2,\beta_b^2),\qquad
 \beta_b^2=(\beta_b^2,\beta_c^2),\qquad
 \beta_c^2=(\beta_a^2,\beta_a^2).
\]
All three squares fix the empty word. If they fix every word of
length $n$, these recursions show that they fix every word of
length $n+1$. Simultaneous induction proves
$\beta_a^2=\beta_b^2=\beta_c^2=1$ on the whole rooted tree.

\subsection{Level Schreier graphs and adjacency operators}

The directed Aleshin level graph has an edge $w\longrightarrow s(w)$
for each $w\in V_n$ and $s\in\{a,b,c\}$. Its indegree and outdegree
are three. Let $A_n$ denote its undirected version, retaining all
labelled edges and their multiplicities. Equivalently, $A_n$ is the
Schreier graph with the symmetric generating list
\[
                    a,b,c,a^{-1},b^{-1},c^{-1}.
\]
It is a degree-six multigraph. A loop contributes two to the degree;
parallel edges are counted separately.

For a permutation $s$ of $V_n$, define its column permutation matrix by
$P_s e_w=e_{s(w)}$, where $e_w$ is a standard basis vector.
Thus $P_sP_t=P_{st}$ and $P_s^*=P_{s^{-1}}$, with $*$ denoting
the conjugate transpose. Set
\begin{equation}\label{eq:adjacency}
 N_n=P_a+P_b+P_c,\qquad
 K_n=N_n+N_n^*,\qquad
 P_n=\frac{K_n}{6}.
\end{equation}
Here $K_n$ is the undirected adjacency matrix and $P_n$ is its
normalized random-walk operator. We use the Euclidean inner product
on $\C^{V_n}$ and write
\[
 H_n=\one^\perp
     =\left\{f\in\C^{V_n}:\sum_{w\in V_n}f(w)=0\right\}.
\]

The Bellaterra level graph $B_n$ has the same vertex set $V_n$ and
one edge labelled $s$ from $w$ to $\beta_s(w)$ for each
$s\in\{a,b,c\}$. Since the Bellaterra states are involutions, its
adjacency matrix is symmetric:
\begin{equation}\label{eq:M}
 M_n=P_{\beta_a}+P_{\beta_b}+P_{\beta_c}=M_n^*.
\end{equation}
A fixed point of a state contributes one to the corresponding
diagonal entry. These fixed points are sometimes represented by
half-loops; the matrix convention fixes their contribution
unambiguously. Every row sum of $M_n$ is three. All labels and
edge multiplicities are retained in both graph families.

\begin{definition}
A sequence of regular graphs is a family of \emph{two-sided spectral
expanders} if its normalized adjacency matrices satisfy
$\norm{P_n|_{H_n}}\leq1-\varepsilon$ for a constant
$\varepsilon>0$ independent of $n$.
\end{definition}

Since $P_n$ is self-adjoint, this bounds the absolute value of every
nonconstant normalized eigenvalue. In particular, it bounds the
eigenvalues away from both $1$ and $-1$.

\section{The exact Bellaterra comparison}
Before comparing the operators, we explain what the states in this
comparison represent. The names \emph{Aleshin} and \emph{Bellaterra}
refer to the two specific three-state binary Mealy automata in
Figure~\ref{fig:automata}. Their states are denoted by
\[
 \{a,b,c\}
 \qquad\text{and}\qquad
 \{\beta_a,\beta_b,\beta_c\},
\]
respectively. Each state specifies two instructions, one for input
$0$ and one for input $1$: which letter to output and which state
to use next. Choosing an initial state therefore determines how
the automaton transforms an entire input word.

For example, starting the Aleshin automaton in state $a$ means
that the first input letter is processed by the rule for $a$.
After that step, the remaining letters are processed in the states
specified by the transition rules. The notation $a(w)$ denotes
the complete output word obtained in this way. Thus the same
symbol $a$ denotes both an internal state and the transformation
of binary words defined by using that state initially. The
symbols $b,c$ and $\beta_a,\beta_b,\beta_c$ have the same
interpretation.

These transformations preserve word length and the prefix relation.
Since each state's output rule permutes $\{0,1\}$, they are
automorphisms of the rooted binary tree. They generate the groups
\[
 \mathcal A=\langle a,b,c\rangle
 \leq\operatorname{Aut}(X^*),
 \qquad
 \mathcal B=\langle\beta_a,\beta_b,\beta_c\rangle
 \leq\operatorname{Aut}(X^*).
\]
Products in these groups are compositions of transformations of
whole words: $(ab)(w)=a(b(w))$ means first run the automaton from
state $b$ on $w$, then run it from state $a$ on the resulting word.
Restricting a state transformation to $V_n=\{0,1\}^n$ gives a
permutation of the $2^n$ words at level $n$. We use the same symbol
for this restriction when the level is clear.

The correspondence $a\leftrightarrow\beta_a$,
$b\leftrightarrow\beta_b$, $c\leftrightarrow\beta_c$ can be read
directly from the following transition table. Each pair in the third
column gives the next Aleshin and Bellaterra states.

\begin{center}
\renewcommand{\arraystretch}{1.12}
\begin{tabular}{@{}ccccc@{}}
\toprule
Current states & Input & Next states
 & Aleshin output & Bellaterra output\\
\midrule
$a,\beta_a$ & $0$ & $c,\beta_c$ & $1$ & $0$\\
$a,\beta_a$ & $1$ & $b,\beta_b$ & $0$ & $1$\\
$b,\beta_b$ & $0$ & $b,\beta_b$ & $1$ & $0$\\
$b,\beta_b$ & $1$ & $c,\beta_c$ & $0$ & $1$\\
$c,\beta_c$ & $0$ & $a,\beta_a$ & $0$ & $1$\\
$c,\beta_c$ & $1$ & $a,\beta_a$ & $1$ & $0$\\
\bottomrule
\end{tabular}
\end{center}

In every row, the next states are again a corresponding pair,
and the two output letters are complementary. Consequently,
when the automata start in corresponding states and read the
same input word, they remain in corresponding states throughout
the calculation. For instance, on input $0010$ the Aleshin
states used are $a,c,a,b$, and
\[
 a(0010)=1c(010)=10a(10)=100b(0)=1001.
\]
The corresponding Bellaterra calculation is
\[
 \beta_a(0010)
 =0\beta_c(010)=01\beta_a(10)=011\beta_b(0)=0110.
\]
The output words differ at every position.

The vertices of the level graphs are the input words. Applying a
state transformation gives the endpoint of the edge with that label.
Thus $0010$ and $1001$ are joined by an $a$-labelled directed Aleshin
edge at level four, while the corresponding Bellaterra edge joins
$0010$ to $0110$. The permutation matrix $P_s$ records the action
of the Aleshin state $s$ by $P_se_w=e_{s(w)}$; the matrix
$P_{\beta_s}$ records the matching Bellaterra action. Here $e_w$
is the coordinate vector indexed by the word $w$. The next lemma
expresses the complementary-output relation as an identity of
permutations and then of matrices.

\subsection{The comparison of states}

Let $\delta_n$ complement every letter:
\[
 \delta_n(w_1\cdots w_n)=(1-w_1)\cdots(1-w_n),
 \qquad D_n=P_{\delta_n}.
\]
Then $D_n^*=D_n$, $D_n^2=I$, and $D_n$ fixes the constant vector.

\begin{lemma}[State comparison]\label{lem:comparison}
For every $n$ and $s\in\{a,b,c\}$,
\begin{equation}\label{eq:states}
                  s=\delta_n\beta_s\quad\text{on }V_n.
\end{equation}
Therefore
\begin{equation}\label{eq:operators}
              N_n=D_nM_n,\qquad
              N_n^*=M_nD_n,\qquad
              K_n=D_nM_n+M_nD_n.
\end{equation}
\end{lemma}

\begin{proof}
The two automata have the same state transition for each input letter.
Their output letters are opposite for every state and input.
Starting in matching states and reading the same word, they therefore
follow the same sequence of states and output complementary words.
This proves \eqref{eq:states}. To see the induction explicitly,
for a word $w$ of length $n-1$ one has
\[
 \delta_n\beta_a(0w)
 =\delta_n\bigl(0\beta_c(w)\bigr)
 =1\delta_{n-1}\beta_c(w)
 =1c(w)=a(0w).
\]
The other states and input letters follow from the same recursions;
the initial case is the empty word. The identity is also
\cite[Lemma~5.1]{VV}. The matrix identities follow from
$P_{st}=P_sP_t$ and $M_n^*=M_n$, $D_n^*=D_n$.
\end{proof}

\subsection{General transfer of spectral gaps}\label{cmp:subsection}

The relation between the two automata transfers any two-sided
Bellaterra spectral bound to the Aleshin graphs. We give the precise
comparison.
Write
\[
 R_n=\frac{M_n}{3},\qquad Q_n=\frac{N_n}{3},\qquad
 P_n=\frac{K_n}{6}=\frac{Q_n+Q_n^*}{2}.
\]
Thus $R_n$ is the normalized Bellaterra adjacency operator, $Q_n$
describes the directed Aleshin walk, and $P_n$ describes the undirected
Aleshin walk. The normalizations use the row sums $3$ and $6$
specified in the graph definitions.

We work on
\[
 H_n=\left\{f:V_n\longrightarrow\mathbb C:
                  \sum_{v\in V_n}f(v)=0\right\},\qquad
 \langle f,g\rangle=\sum_{v\in V_n}\overline{f(v)}g(v).
\]
For a permutation $s$, our convention $P_se_v=e_{s(v)}$ gives
\[
 P_sP_t=P_{st},\qquad P_s^*=P_{s^{-1}}.
\]
Every permutation matrix is unitary and preserves both constants and
$H_n$, because permuting coordinates preserves their sum and their
squared absolute values. In particular, these subspaces are preserved
by the operators above and their adjoints.

\begin{proposition}\label{cmp:singular}
For $s,t\in\{a,b,c\}$, the corresponding automaton transformations
satisfy
\begin{equation}\label{cmp:products}
 \beta_s\beta_t=s^{-1}t.
\end{equation}
If $D_n$ is the permutation matrix of simultaneous bit complementation,
then
\begin{gather}
 Q_n=D_nR_n,\qquad Q_n^*=R_nD_n,
 \qquad P_n=\frac{D_nR_n+R_nD_n}{2},
 \label{cmp:factorization}\\
 Q_n^*Q_n=R_n^2,\qquad Q_nQ_n^*=D_nR_n^2D_n.
 \label{cmp:squares}
\end{gather}
Consequently, the singular values of $Q_n|_{H_n}$ are precisely the
absolute values of the eigenvalues of $R_n|_{H_n}$, with multiplicity.
In particular,
\begin{equation}\label{cmp:exact-norm}
 \|Q_n|_{H_n}\|=\|Q_n^*|_{H_n}\|=\|R_n|_{H_n}\|.
\end{equation}
\end{proposition}

\begin{proof}
Let $\delta_n$ complement every letter of a word in $V_n$.
Lemma~\ref{lem:comparison} gives $s=\delta_n\beta_s$. The Bellaterra
states and $\delta_n$ are involutions. Hence
\[
 s^{-1}t=(\delta_n\beta_s)^{-1}(\delta_n\beta_t)
        =\beta_s\delta_n^2\beta_t
        =\beta_s\beta_t,
\]
which proves \eqref{cmp:products}. The identities are compatible with
restriction to every finite level, so they hold on the rooted tree
as well. Every product of two Bellaterra generators is thus a word
of length at most two in the Aleshin generators and their inverses.

The involution properties give
\[
 R_n^*=R_n,\qquad D_n^*=D_n,\qquad D_n^2=I.
\]
Summing $P_s=D_nP_{\beta_s}$ and dividing by $3$ proves
$Q_n=D_nR_n$. Taking the adjoint gives $Q_n^*=R_nD_n$, and averaging
these identities proves \eqref{cmp:factorization}. Multiplication
then gives \eqref{cmp:squares}. The first identity in
\eqref{cmp:squares} also follows directly from \eqref{cmp:products}:
\[
 Q_n^*Q_n
 =\frac19\sum_{s,t\in\{a,b,c\}}P_{s^{-1}t}
 =\frac19\sum_{s,t\in\{a,b,c\}}P_{\beta_s\beta_t}
 =R_n^2.
\]

The \emph{singular values} of an operator $T$ are the nonnegative
square roots of the eigenvalues of $T^*T$. All the operators in
\eqref{cmp:squares} preserve $H_n$, so the identities hold on that
subspace. Choose an orthonormal eigenbasis for the self-adjoint
operator $R_n|_{H_n}$. For an eigenvector $f$ with eigenvalue $\lambda$,
\[
 Q_n^*Q_nf=R_n^2f=\lambda^2f.
\]
The corresponding singular value is therefore $|\lambda|$, proving
the assertion with multiplicities.

For the norm equality one can also argue directly. The unitarity of
$D_n$ gives
\[
 \|Q_nf\|=\|D_nR_nf\|=\|R_nf\|\qquad(f\in H_n).
\]
Taking the supremum over unit vectors proves
$\|Q_n|_{H_n}\|=\|R_n|_{H_n}\|$. An operator and its adjoint have
the same norm. Since $H_n$ is preserved by both $Q_n$ and $Q_n^*$,
the adjoint of the restriction is the restriction of the adjoint.
This proves the remaining equality in \eqref{cmp:exact-norm}.
\end{proof}

For a self-adjoint normalized adjacency operator $T$, its
\emph{two-sided spectral gap} is
\[
 \varepsilon(T)=1-\|T|_{\mathbf1^\perp}\|.
\]
Thus a gap of at least $\varepsilon$ means that all eigenvalues on
the orthogonal complement of the constant functions lie in
$[-1+\varepsilon,1-\varepsilon]$. Both ends of the spectrum are
controlled by this definition.

\begin{proposition}
\label{cmp:transfer}
At every level,
\begin{equation}\label{cmp:norm-bound}
 \|P_n|_{H_n}\|\leq\|R_n|_{H_n}\|.
\end{equation}
Equivalently, the normalized gaps satisfy
\begin{equation}\label{cmp:gap-bound}
 \varepsilon_n^{\mathrm A}=1-\|P_n|_{H_n}\|
 \geq
 \varepsilon_n^{\mathrm B}=1-\|R_n|_{H_n}\|.
\end{equation}
In particular, a positive Bellaterra two-sided gap uniform in $n$
gives an Aleshin two-sided gap with the same normalized constant.
Whenever this gap is positive, the undirected Aleshin graph is
connected and nonbipartite.
\end{proposition}

\begin{proof}
Put $\rho_n=\|R_n|_{H_n}\|$. For $f\in H_n$, the triangle inequality
and Proposition~\ref{cmp:singular} give
\[
 \|P_nf\|
 =\left\|\frac{Q_nf+Q_n^*f}{2}\right\|
 \leq\frac{\|Q_nf\|+\|Q_n^*f\|}{2}
 \leq\rho_n\|f\|.
\]
Taking the supremum over unit vectors proves \eqref{cmp:norm-bound},
and subtracting from $1$ proves \eqref{cmp:gap-bound}.

There is also a useful proof directly in terms of the eigenvalues.
By \eqref{cmp:factorization} and our inner-product convention,
\[
 \langle f,P_nf\rangle
 =\operatorname{Re}\langle D_nf,R_nf\rangle.
\]
Cauchy--Schwarz and the unitarity of $D_n$ yield
\begin{equation}\label{cmp:quadratic}
 |\langle f,P_nf\rangle|
 \leq\|D_nf\|\,\|R_nf\|
 \leq\rho_n\|f\|^2\qquad(f\in H_n).
\end{equation}
Substituting an eigenvector of the self-adjoint operator $P_n|_{H_n}$
shows that
\begin{equation}\label{cmp:spectral-interval}
 \operatorname{Spec}(P_n|_{H_n})\subseteq[-\rho_n,\rho_n].
\end{equation}
If $\rho_n\leq1-\varepsilon$ uniformly, the same $\varepsilon$ therefore
works for every Aleshin level. In unnormalized terms, a Bellaterra
bound $|\lambda|\leq3-\eta$ gives the Aleshin bound
$|\lambda|\leq6-2\eta$ on $H_n$.

Suppose the undirected Aleshin graph had a component $S$ with
$0<|S|<|V_n|$. Then
\[
 f=\mathbf1_S-\frac{|S|}{|V_n|}\mathbf1
\]
would be a nonzero vector in $H_n$ with $P_nf=f$. This contradicts
\eqref{cmp:spectral-interval} when $\rho_n<1$, proving connectedness.
If the graph were bipartite, the sign function $f$ equal to $1$ and
$-1$ on the two parts would satisfy $P_nf=-f$. Self-adjointness and
$P_n\mathbf1=\mathbf1$ imply
\[
 \langle\mathbf1,f\rangle
 =\langle P_n\mathbf1,f\rangle
 =\langle\mathbf1,P_nf\rangle
 =-\langle\mathbf1,f\rangle,
\]
so $f\in H_n$, again contradicting \eqref{cmp:spectral-interval}.
\end{proof}

Conjectures~1.3 and~1.2 of Malyshev--Pak \cite{MP} concern two-sided
expansion of the Bellaterra and Aleshin graph families, respectively.
In \cite[Remark~11.3]{MP}, they observe that the first implies the
second by using products of two Bellaterra generators.
Identity~\eqref{cmp:products} describes those products explicitly,
and Proposition~\ref{cmp:transfer} gives the quantitative implication.
Their directed Aleshin graph records the three outgoing edges at
each vertex; forgetting the directions, with multiplicities retained,
gives the normalized operator $P_n$ used here. Their automaton
convention is obtained from ours by interchanging $0$ and $1$ in
both the input and output labels. This conjugates the level actions
by $\delta_n$ and gives isomorphic labelled level graphs.

The two-sided hypothesis is essential to this general argument.
For a self-adjoint $R_n$, the norm on $H_n$ is the maximum
\emph{absolute value} of its eigenvalues there; an upper bound on the
largest eigenvalue alone does not control eigenvalues near $-1$.
For example, on a two-point set let
\[
 R=D=\begin{pmatrix}0&1\\1&0\end{pmatrix},\qquad
 H=\operatorname{span}\{(1,-1)\}.
\]
The eigenvalue of $R|_H$ is $-1$, so its one-sided gap is $2$.
Nevertheless, $(DR+RD)/2=I$ has gap zero on $H$.
This example shows why the general comparison requires a bound at
the negative end of the spectrum as well.

The exact relation in \eqref{cmp:squares} identifies singular values,
and the eigenvalues themselves may differ. Already at level one,
\[
 R_1=\frac13\begin{pmatrix}2&1\\1&2\end{pmatrix},\qquad
 P_1=\frac13\begin{pmatrix}1&2\\2&1\end{pmatrix}.
\]
Their eigenvalues on $H_1=\operatorname{span}\{(1,-1)\}$ are
$1/3$ and $-1/3$, respectively, although both restricted norms are
$1/3$.

\subsection{The Bellaterra bound and Aleshin expansion}

We now apply Proposition~\ref{cmp:transfer} using the Bellaterra
spectral theorem from 
\cite[Theorem~2.3]{Bell}. 

\begin{theorem}[Bellaterra spectral bound]\label{thm:bellaterra}
For every $n\geq1$, the graph $B_n$ is connected, the eigenvalue $3$
of $M_n$ is simple, and
\begin{equation}\label{eq:input}
 \norm{M_n|_{H_n}}\leq2\sqrt2.
\end{equation}
\end{theorem}

\begin{proof}
Write $\gamma_a,\gamma_b,\gamma_c$ for the Bellaterra states in \cite{Bell}. They satisfy
\[
 \gamma_a=(\gamma_b,\gamma_c),\qquad
 \gamma_b=(\gamma_c,\gamma_b),\qquad
 \gamma_c=(\gamma_a,\gamma_a)\sigma.
\]
Conjugating the states in \eqref{eq:bellaterra} by letter
complementation interchanges the sections at $0$ and $1$ and leaves
the first-letter permutation unchanged. Thus the conjugated states
satisfy precisely these recursions. For example, if $w\in V_{n-1}$,
then
\[
 \begin{aligned}
 (\delta_n\beta_a\delta_n)(0w)
 &=\delta_n\beta_a(1\delta_{n-1}(w))\\
 &=\delta_n\bigl(1\beta_b(\delta_{n-1}(w))\bigr)\\
 &=0(\delta_{n-1}\beta_b\delta_{n-1})(w).
 \end{aligned}
\]
The other transitions follow in the same way. Starting with the
empty word, simultaneous induction on word length therefore gives
\[
          \gamma_s=\delta_n\beta_s\delta_n
          \quad\text{on }V_n,\qquad s\in\{a,b,c\}.
\]
Our paper \cite{Bell} uses the same edge-multiplicity and fixed-point
conventions as \eqref{eq:M}. Its adjacency operator is consequently
\[
 \widehat M_n=P_{\gamma_a}+P_{\gamma_b}+P_{\gamma_c}
             =D_nM_nD_n.
\]
Since $D_n$ is unitary and fixes the constant vector, it preserves
$H_n$. The estimate in \cite[Theorem~2.3]{Bell} therefore yields
\[
 \norm{M_n|_{H_n}}=\norm{\widehat M_n|_{H_n}}\leq2\sqrt2.
\]
Finally, $2\sqrt2<3$ excludes an eigenvector of eigenvalue $3$ in
$H_n$. The eigenvalue $3$ is simple, so $B_n$ is connected.
\end{proof}

\begin{samepage}
\begin{theorem}[Aleshin spectral expansion]\label{thm:expansion}
For every $n\geq1$,
\begin{equation}\label{eq:bound}
 \Spec(K_n|_{H_n})\subseteq[-4\sqrt2,4\sqrt2],
 \qquad \norm{P_n|_{H_n}}\leq\frac{2\sqrt2}{3}<1.
\end{equation}
Thus the degree-six Aleshin graphs $A_n$ are connected two-sided
expanders, with uniform normalized spectral gap at least
\begin{equation}\label{eq:gap}
 \varepsilon=1-\frac{2\sqrt2}{3}.
\end{equation}
\end{theorem}
\end{samepage}

\begin{proof}
Theorem~\ref{thm:bellaterra} gives
$\norm{R_n|_{H_n}}\leq2\sqrt2/3$.
Proposition~\ref{cmp:transfer} therefore yields
$\norm{P_n|_{H_n}}\leq2\sqrt2/3<1$.
Since $K_n=6P_n$, its spectrum on $H_n$ lies in
$[-4\sqrt2,4\sqrt2]$. The same proposition proves connectedness,
and the normalized two-sided gap is at least
$1-2\sqrt2/3$, independently of $n$.
\end{proof}

The general comparison in Proposition~\ref{cmp:transfer} uses only
the automaton identities and elementary linear algebra. The
particular constants in Theorem~\ref{thm:expansion} use the
Bellaterra estimate \eqref{eq:input}. We next derive the edge
expansion and diameter bounds from these operators.

\section{Edge expansion and the directed walk}

For $S\subset V_n$, let $\partial S$ be the multiset of undirected
edges with one endpoint in $S$ and one outside $S$. Count each such
edge once, with its multiplicity; loops contribute nothing.

\begin{corollary}[Edge expansion]\label{cor:edge}
Under \eqref{eq:input}, every $S\subset V_n$ with
$0<|S|\leq2^{n-1}$ satisfies
\begin{equation}\label{eq:edge}
                    {|\partial S|\geq(3-2\sqrt2)|S|.}
\end{equation}
In particular, the edge-expansion constants obey
\[
 h(A_n):=\min_{0<|S|\leq|V_n|/2}\frac{|\partial S|}{|S|}
                              \geq3-2\sqrt2>0.
\]
\end{corollary}

\begin{proof}
Put $m=2^n$, $s=|S|$, and
$f=\one_S-(s/m)\one$. Then $f\in H_n$ and
\[
                        \norm{f}^2=s\left(1-\frac{s}{m}\right).
\]
The graph Laplacian is $L_n=6I-K_n$, and $L_n\one=0$. The usual
quadratic-form identity, valid with the stated multiplicity convention,
gives
\[
 |\partial S|
 =\langle\one_S,L_n\one_S\rangle
 =\langle f,L_nf\rangle
 =\frac12\sum_{u,v\in V_n}(K_n)_{uv}|f(u)-f(v)|^2.
\]
More generally, if $\norm{P_n|_{H_n}}\leq1-\varepsilon$,
then $L_n=6(I-P_n)$ gives
\begin{align*}
 |\partial S|
 &\geq6\varepsilon\norm{f}^2\\
 &=6\varepsilon s\left(1-\frac{s}{m}\right)
 \geq3\varepsilon s,
\end{align*}
where the last step uses $s\leq m/2$. By
Theorem~\ref{thm:expansion}, we may take
$\varepsilon=1-2\sqrt2/3$. This gives the stated constant
$3\varepsilon=3-2\sqrt2$.
\end{proof}

For completeness, the positive-generator directed walk also has
a uniform contraction estimate. With our column convention,
$Q_n=N_n/3$ acts on column probability distributions, while its adjoint
acts on observables by
\[
             (Q_n^*f)(w)=\frac{f(a(w))+f(b(w))+f(c(w))}{3}.
\]
Both preserve the uniform distribution and the mean-zero subspace.
Proposition~\ref{cmp:singular} gives
\begin{equation}\label{eq:directed}
       \norm{Q_n|_{H_n}}
        =\norm{Q_n^*|_{H_n}}
        =\norm{(M_n/3)|_{H_n}}
        \leq\frac{2\sqrt2}{3}.
\end{equation}
Thus the directed walk contracts mean-zero vectors by the same
uniform factor in a single step. For every integer $t\geq0$,
submultiplicativity of the operator norm gives
\[
 \norm{Q_n^t|_{H_n}}
 \leq\norm{Q_n|_{H_n}}^t
 \leq\left(\frac{2\sqrt2}{3}\right)^t.
\]
This argument does not require $Q_n$ to be self-adjoint or
diagonalizable. For instance, if $\mu$ is a probability
column vector and $u=m^{-1}\one$, then
\[
           \norm{Q_n^t\mu-u}
                 \leq\left(\frac{2\sqrt2}{3}\right)^t\norm{\mu-u}.
\]
We use this contraction estimate for directed paths in the next section.

\section{Linear diameter bounds for both families}\label{sec:diameters}

The \emph{distance} between two vertices of an undirected graph is the
minimum number of edges in a path joining them. The \emph{diameter}
is the maximum of these distances. For a directed graph, paths must
follow the orientations; the \emph{directed diameter} is the maximum
over ordered pairs of vertices. We write
$\operatorname{diam}_{\mathrm{dir}}(A_n)$ for the directed Aleshin
diameter, using only the transitions $a,b,c$.

Malyshev--Pak's quadratic diameter bounds
\cite[Theorem~1.1]{MP} concern undirected distances. We obtain linear
bounds for those distances and also for directed Aleshin distances.
The argument for undirected graphs uses the following elementary
polynomial estimate.

\subsection{A polynomial bound for the diameter}

The Chebyshev polynomials $T_j$ are defined by
\[
 T_0(x)=1,\qquad T_1(x)=x,\qquad
 T_{j+1}(x)=2xT_j(x)-T_{j-1}(x).
\]
They have degree $j$ and satisfy
\begin{equation}\label{eq:chebyshev}
 T_j(\cos\theta)=\cos(j\theta),\qquad
 T_j(\cosh u)=\cosh(ju).
\end{equation}
These identities follow by induction from the corresponding addition
formulas. In particular, $|T_j(x)|\leq1$ when $-1\leq x\leq1$.

\begin{lemma}[Diameter from a spectral interval]\label{lem:polydiameter}
Let $P$ be the symmetric normalized adjacency matrix of a regular
graph on $N\geq2$ vertices, with loops and multiplicities allowed.
Suppose that
\[
              \norm{P|_{\one^\perp}}\leq\rho,
                       \qquad 0<\rho<1.
\]
If an integer $t\geq1$ satisfies
\begin{equation}\label{eq:chebcriterion}
                         T_t(1/\rho)>N-1,
\end{equation}
then the graph has diameter at most $t$.
\end{lemma}

\begin{proof}
Set
\[
 J=\frac1N\one\one^{\mathsf T},\qquad
 a_t=T_t(1/\rho),\qquad
 R=T_t(P/\rho)-a_tJ.
\]
The operator $R$ vanishes on constants. On $\one^\perp$, the spectrum
of $P/\rho$ lies in $[-1,1]$. The spectral theorem and
\eqref{eq:chebyshev} therefore give $\norm{R}\leq1$.

For a vertex $v$, put $f_v=e_v-N^{-1}\one$. Then
$\norm{f_v}^2=1-N^{-1}$. Since $R$ vanishes on constants and is
self-adjoint, the Cauchy--Schwarz inequality gives
\[
 |R_{uv}|=|\langle f_u,Rf_v\rangle|
                  \leq\norm{f_u}\norm{f_v}=1-\frac1N.
\]
Consequently
\begin{equation}\label{eq:positivepoly}
 \bigl(T_t(P/\rho)\bigr)_{uv}
       \geq\frac{a_t}{N}-\left(1-\frac1N\right)
       =\frac{a_t-(N-1)}{N}>0.
\end{equation}

If $u,v$ had distance greater than $t$, every entry
$(P^j)_{uv}$ with $0\leq j\leq t$ would be zero, because there is
no walk of any such length joining them. As $T_t(P/\rho)$ is a
polynomial in $P$ of degree $t$, its $uv$-entry would also be zero.
This contradicts \eqref{eq:positivepoly}. Thus every pair of vertices
has distance at most $t$.
\end{proof}

\subsection{The Bellaterra and undirected Aleshin bounds}

\begin{theorem}[Linear undirected diameter]\label{thm:diameter}
The Bellaterra estimate \eqref{eq:input} gives, for every $n\geq1$,
\begin{equation}\label{eq:lineardiameter}
 {\operatorname{diam}(B_n)\leq2n+2,\qquad
        \operatorname{diam}(A_n)\leq2n+2.}
\end{equation}
\end{theorem}

\begin{proof}
For the two families, the symmetric normalized adjacency operators are
\[
              P_n^{\mathrm B}=\frac{M_n}{3},\qquad
              P_n^{\mathrm A}=\frac{K_n}{6}.
\]
Equation~\eqref{eq:input} and Theorem~\ref{thm:expansion} give the
same bound for both operators:
\[
 \norm{P_n^{\mathrm B}|_{H_n}}\leq\rho,
 \qquad \norm{P_n^{\mathrm A}|_{H_n}}\leq\rho,
 \qquad \rho=\frac{2\sqrt2}{3}.
\]
There are $N=2^n$ vertices. Since
\[
                 \frac1\rho=\frac{3}{2\sqrt2}
                                  =\cosh(\log\sqrt2),
\]
the second identity in \eqref{eq:chebyshev} gives
\[
                   T_t(1/\rho)=\frac{2^{t/2}+2^{-t/2}}2.
\]
In particular,
\[
                 T_{2n+2}(1/\rho)
                      =N+\frac1{4N}>N-1.
\]
Apply Lemma~\ref{lem:polydiameter} with $t=2n+2$ to each operator.
\end{proof}

\subsection{Directed Aleshin paths}

\begin{proposition}[Linear directed diameter]\label{prop:directeddiameter}
The Bellaterra estimate \eqref{eq:input} gives, for every $n\geq1$,
\begin{equation}\label{eq:directeddiameter}
                  {\operatorname{diam}_{\mathrm{dir}}(A_n)
                         \leq12n.}
\end{equation}
\end{proposition}

\begin{proof}
Use $Q_n=N_n/3$, $N=2^n$, $J=N^{-1}\one\one^{\mathsf T}$,
and $\rho=2\sqrt2/3$. The matrix $Q_n$ is doubly stochastic: its
row sums and column sums are one. Hence $Q_nJ=JQ_n=J$.
The contraction estimate \eqref{eq:directed} gives, for every integer
$t\geq1$,
\[
                 \norm{Q_n^t-J}
                   =\norm{(Q_n-J)^t}\leq\rho^t.
\]
Every matrix entry therefore satisfies
\[
                         (Q_n^t)_{vu}\geq\frac1N-\rho^t.
\]
The exact inequality
\[
           \rho^{12}=\left(\frac89\right)^6
                       =\frac{262144}{531441}<\frac12
\]
shows that $(Q_n^{12n})_{vu}>0$ for every ordered pair $u,v$.
With our column convention this is a directed walk of length $12n$
from $u$ to $v$. The directed diameter is at most $12n$.
\end{proof}

\subsection{Lower bounds and the order of growth}

\begin{corollary}[The diameters have linear order]\label{cor:diameterorder}
Under the Bellaterra estimate \eqref{eq:input}, for every $n\geq1$,
\begin{align}
 \log_2\!\left(\frac{2^n+2}{3}\right)
    &\leq\operatorname{diam}(B_n)\leq2n+2,
       \label{eq:bellaterralower}\\
 \log_5\!\left(\frac{2^{n+1}+1}{3}\right)
    &\leq\operatorname{diam}(A_n)\leq2n+2,
       \label{eq:aleshinlower}\\
 \log_3(2^{n+1}+1)-1
    &\leq\operatorname{diam}_{\mathrm{dir}}(A_n)\leq12n.
       \label{eq:directedlower}
\end{align}
In particular, all three diameters are $\Theta(n)$, equivalently
$\Theta(\log |V_n|)$.
\end{corollary}

\begin{proof}
In an undirected graph with at most $d\geq3$ neighbors at each vertex,
a ball of radius $D$ contains at most
\[
                         1+d\sum_{j=0}^{D-1}(d-1)^j
\]
vertices: there are at most $d$ choices for the first step of a
shortest path, and at most $d-1$ choices for each later step.
Loops and multiple edges cannot increase the number of distinct
vertices reached. For Bellaterra, $d=3$, so a graph of diameter $D$
has at most $3\cdot2^D-2$ vertices. For undirected Aleshin, $d=6$,
giving at most $(3\cdot5^D-1)/2$ vertices. Substituting $2^n$ for the
number of vertices and solving for $D$ gives the first two lower bounds.

The directed Aleshin graph has three outgoing edges at each vertex.
At most $1+3+\cdots+3^D=(3^{D+1}-1)/2$ vertices can be reached
by directed paths of length at most $D$. This gives the third lower
bound. Combine these estimates with
Theorem~\ref{thm:diameter} and Proposition~\ref{prop:directeddiameter}.
Here $\Theta(n)$ means that the quantity is bounded above and below
by positive constant multiples of $n$ for all sufficiently large $n$.
\end{proof}

% Section 5. Requires the main note's theorem environments and macros.
% External label: thm:expansion. No additional bibliography entries.

\section{The Aleshin graphs are not Ramanujan from level ten}
\label{sec:nonramanujan}

A connected $d$-regular graph is \emph{Ramanujan} if every adjacency
eigenvalue other than the trivial values $d$ and, when present, $-d$
has absolute value at most $2\sqrt{d-1}$. We use this same spectral
definition for the multigraphs in the note, with loops and edge
multiplicities counted in their adjacency matrices. For the
degree-six Aleshin graphs, the Ramanujan threshold is $2\sqrt5$.
The expansion bound $4\sqrt2$ from Theorem~\ref{thm:expansion}
is larger than this threshold. The following argument exhibits an
eigenvalue above $2\sqrt5$ at level ten and carries it to every
higher level.

\subsection{Eigenvalues persist under passage to higher levels}

For $n\geq m$, let $\pi_{n,m}:V_n\to V_m$ delete the last $n-m$
letters, and define the pullback of functions by
\[
 (L_{n,m}f)(w)=f(\pi_{n,m}(w)).
\]
Each word of length $m$ has exactly $2^{n-m}$ extensions of length
$n$, so
\begin{equation}\label{eq:level-pullback-norm}
 \|L_{n,m}f\|^2=2^{n-m}\|f\|^2,
 \qquad
 \sum_{w\in V_n}(L_{n,m}f)(w)
       =2^{n-m}\sum_{v\in V_m}f(v).
\end{equation}
In particular, $L_{n,m}$ is injective and carries $H_m$ into $H_n$.

\begin{lemma}[Persistence of level eigenvalues]
\label{lem:level-eigenvalues}
For $n\geq m$,
\begin{equation}\label{eq:level-intertwining}
 K_nL_{n,m}=L_{n,m}K_m.
\end{equation}
Consequently, every eigenvalue of $K_m|_{H_m}$ is an eigenvalue
of $K_n|_{H_n}$.
\end{lemma}

\begin{proof}
A rooted-tree automorphism preserves prefixes. Thus, for each
$s\in\{a,b,c,a^{-1},b^{-1},c^{-1}\}$,
\[
 \pi_{n,m}(s(w))=s(\pi_{n,m}(w)).
\]
Using $(P_sf)(w)=f(s^{-1}(w))$, this gives
$P_sL_{n,m}=L_{n,m}P_s$, where the two occurrences of $P_s$ act
at their respective levels. Summing over the six generators proves
\eqref{eq:level-intertwining}. If $K_mf=\lambda f$ with
$0\ne f\in H_m$, then $0\ne L_{n,m}f\in H_n$ by
\eqref{eq:level-pullback-norm}, and
$K_nL_{n,m}f=\lambda L_{n,m}f$.
\end{proof}

\subsection{An exact certificate at level ten}

\begin{theorem}\label{thm:aleshin-nonramanujan}
For every $n\geq10$, the undirected Aleshin graph $A_n$ is not
Ramanujan. More precisely, its adjacency matrix has a nonconstant
eigenvalue $\lambda$ satisfying
\[
 \lambda\geq
 \frac{572251333403184}{127833554901040}
 >\frac{179}{40}>2\sqrt5.
\]
\end{theorem}

\begin{proof}
We give a finite certificate consisting entirely of integer matrix
recursions and exact inner products. Order the words at each level
lexicographically, with words starting with $0$ before those
starting with $1$. Let $\mathsf A_m,\mathsf B_m,\mathsf C_m$
be the column permutation matrices of $a,b,c$ on $V_m$.
At level zero they are all the $1\times1$ matrix $(1)$.
The wreath recursions \eqref{eq:aleshin} give
\begin{equation}\label{eq:level10-blocks}
 \mathsf A_{m+1}=
 \begin{pmatrix}0&\mathsf B_m\\ \mathsf C_m&0\end{pmatrix},
 \quad
 \mathsf B_{m+1}=
 \begin{pmatrix}0&\mathsf C_m\\ \mathsf B_m&0\end{pmatrix},
 \quad
 \mathsf C_{m+1}=
 \begin{pmatrix}\mathsf A_m&0\\0&\mathsf A_m\end{pmatrix}.
\end{equation}
Thus the symmetric $1024\times1024$ integer matrix at level ten is
\[
 K=K_{10}
 =\mathsf A_{10}+\mathsf B_{10}+\mathsf C_{10}
  +\mathsf A_{10}^{\mathsf T}
  +\mathsf B_{10}^{\mathsf T}
  +\mathsf C_{10}^{\mathsf T}.
\]

Let $\chi\in\mathbb Z^{V_{10}}$ be the parity vector
\[
 \chi(x_1\cdots x_{10})=(-1)^{x_1+\cdots+x_{10}}.
\]
It has coordinate sum zero. Define
\begin{equation}\label{eq:level10-vector-recursion}
 u_0=\chi,\qquad u_1=K\chi,\qquad
 u_{j+1}=Ku_j-5u_{j-1}\quad(1\leq j\leq12),
 \qquad v=\frac14(K+4I)u_{13}.
\end{equation}
Since $K$ preserves $H_{10}$, all the $u_j$ and $v$ lie in
$H_{10}$. Their construction is explicit from
\eqref{eq:level10-blocks} and \eqref{eq:level10-vector-recursion}.

The four inner products needed to evaluate the Rayleigh quotient
of $v$ are as follows. Put $\mu_j=\langle u_{13},K^ju_{13}\rangle$.
Successive application of the displayed integer recursions gives
\begin{equation}\label{eq:level10-moments}
\begin{array}{c|r}
 j&\mu_j\\\hline
 0&29\,696\,431\,560\,960\\
 1&122\,512\,824\,426\,880\\
 2&590\,091\,378\,026\,240\\
 3&2\,475\,085\,119\,410\,944
\end{array}
\end{equation}
These are exact integers. Expanding $(K+4I)^2$ and using the
symmetry of $K$ now gives
\begin{align}
 d:=\langle v,v\rangle
 &=\frac{\mu_2+8\mu_1+16\mu_0}{16}
   =127\,833\,554\,901\,040,
       \label{eq:level10-denominator}\\
 q:=\langle v,Kv\rangle
 &=\frac{\mu_3+8\mu_2+16\mu_1}{16}
   =572\,251\,333\,403\,184.
       \label{eq:level10-numerator}
\end{align}
In particular $v\ne0$. The strict comparison with the Ramanujan
threshold follows from the two integer inequalities
\begin{gather*}
 40q-179d=7\,847\,008\,841\,200>0,\\
 179^2-20\cdot40^2=41>0.
\end{gather*}
They imply
\[
 \frac{\langle v,Kv\rangle}{\langle v,v\rangle}
 =\frac qd>\frac{179}{40}>\sqrt{20}=2\sqrt5.
\]

To see why this proves an eigenvalue bound, expand $v$ in an
orthonormal eigenbasis of the real symmetric operator $K|_{H_{10}}$.
Its Rayleigh quotient is a weighted average of the corresponding
eigenvalues, with nonnegative weights summing to one. At least one
eigenvalue is therefore at least $q/d$. Theorem~\ref{thm:expansion}
places every eigenvalue on $H_{10}$ strictly below $6$, so the
eigenvalue just obtained is nontrivial. This proves that $A_{10}$
is not Ramanujan.

Finally, Lemma~\ref{lem:level-eigenvalues} carries this same
eigenvalue to $K_n|_{H_n}$ for every $n\geq10$. It remains strictly
between $2\sqrt5$ and $6$ and hence violates the degree-six
Ramanujan bound at every such level.
\end{proof}

The theorem establishes failure of the Ramanujan bound for every
level from ten onward. This is compatible with the uniform spectral
gap proved earlier: expansion requires a fixed separation from
$6$ and $-6$, whereas the Ramanujan condition imposes the smaller
interval $[-2\sqrt5,2\sqrt5]$ on the nontrivial spectrum.

\section{Tools used.}

The paper was written with ChatGPT 6 Pro editing assistance.
The proof of the fact that the Aleshin graphs are not Ramanujan, Section 5.2,
was done with ChatGPT 6 Pro assistance.

The content of the paper was carefully checked by the authors
and the authors carry the full responsibility of the content of the paper.

\end{document}